\documentclass{amsart}

\usepackage[all]{xy}                        %

\CompileMatrices                            

\UseTips                                    
\input xypic
\usepackage[bookmarks=true]{hyperref}       

\usepackage{amssymb,latexsym,amsmath,amscd}
\usepackage{xspace,enumerate}
\usepackage{graphicx, color}
\usepackage{tikz}

\reversemarginpar

\theoremstyle{plain}
\newtheorem{theorem}{Theorem}[section]
\newtheorem*{theorem*}{Theorem}
\newtheorem{proposition}[theorem]{Proposition}
\newtheorem{corollary}[theorem]{Corollary}

\theoremstyle{definition}

\newtheorem{notation}[theorem]{Notation}

\newtheorem{remark}[theorem]{Remark}

\newtheorem{question}[theorem]{Question}

\begin{document}

\title[A new class of non-identifiable skew symmetric tensors]{A new class of non-identifiable skew symmetric tensors}

\author[A. Bernardi]{Alessandra Bernardi}
\address[A. Bernardi]{Dipartimento di Matematica, Universit\`a di Trento, Via Sommarive 14, Povo (TN), Italy}
\email{alessandra.bernardi@unitn.it}

\author[D. Vanzo]{Davide Vanzo}
\address[D. Vanzo]{Dipartimento di Matematica e Informatica ``Ulisse Dini", Universit\`a di Firenze, Viale Morgagni 67/a, Firenze, Italy}
\email{davide.vanzo@unifi.it}

\maketitle
\begin{abstract} We prove that the generic element of the fifth secant variety $\sigma_5(Gr(\mathbb{P}^2,\mathbb{P}^9)) \subset \mathbb{P}(\bigwedge^3 \mathbb{C}^{10})$ of the Grassmannian of planes of $\mathbb{P}^9$ has exactly two decompositions as a sum of five projective classes of decomposable skew-symmetric tensors. {We show that this, {together with $Gr(\mathbb{P}^3, \mathbb{P}^8)$, is the only non-identifiable case} among the non-defective secant varieties $\sigma_s(Gr(\mathbb{P}^k, \mathbb{P}^n))$ for any $n<14$. In the same range for $n$, we classify all the weakly defective and all tangentially weakly defective secant varieties of any Grassmannians.} We also show that the dual variety $(\sigma_3(Gr(\mathbb{P}^2,\mathbb{P}^7)))^{\vee}$  of the variety of 3-secant planes of the Grassmannian of $\mathbb{P}^2\subset \mathbb{P}^7$  is $\sigma_2(Gr(\mathbb{P}^2,\mathbb{P}^7))$ the variety of bi-secant lines of the same Grassmannian. The proof of this last fact has a very interesting physical interpretation in terms of measurement of the entanglement of a system of 3 identical fermions, the state of each of them belonging to a 8-th dimensional ``Hilbert'' space.
\end{abstract}


\section*{Introduction}

Let  $X \subset \mathbb{P}^n$ be any reduced, irreducible projective variety defined over $\mathbb{C}$. A point $t \in \mathbb{P}^n$ has $X$-rank equal to $r$ if $r$ is the minimum integer for which there exist  $r$ points  $x_1 \ldots , x_r \in X$ such that
\begin{equation}\label{rank}t \in \langle x_1, \ldots , x_r \rangle \end{equation}
$\langle x_1 , \ldots , x_r \rangle\simeq \mathbb{P}^{r-1}$ denotes the projective linear span of the  $x_i$'s. We will also say that in this case $\{x_1, \ldots , x_s\}$ is a \emph{decomposition} of $t$.
The Zariski closure of the set $\{t \in \mathbb{P}^n\, | \, X\hbox{-rank}(t) = r\}$ is the so called \emph{$r$-secant variety}  $\sigma_r(X)$ of $X$.  There is an expected dimension for $\sigma_r(X)$ that is $\mathrm{expdim}\, \sigma_r(X)=\min \{r(\dim X+1)-1, \dim \langle X \rangle \}$. 
 The actual dimension of $\sigma_r(X)$ can be smaller than the expected {as it can be computed by Terracini's Lemma (see e.g. \cite{Te, A})}. When this happens we say that $X$ is \emph{$r$-defective} with \emph{$r$-defect} $\delta=\mathrm{expdim} \, \sigma_r(X)-\dim \sigma_r(X)$.

The \emph{$r$-th secant degree} of  $X$ is the number of $\mathbb{P}^{r-1}$'s containing the generic element  $t \in \sigma_r(X)$ and that  are $r$-secant  to $X$ as in (\ref{rank}).
{Regarding the $r$-th secant degree, when the dimension of $\sigma_r(X)$ is not the expected
   one, it is infinity.} 

The variety $X$ is said to be \emph{$r$-identifiable} if the $r$-th secant degree of  $X$ is equal to 1.

Moreover, a variety $X$ which is not $r$-defective is said to be \emph{perfect} if $(\dim X +1)$ divides  $n+1$. In this case we expect a finite number of decompositions also for a generic  $t \in \mathbb{P}^n$.   Note that the $r$-th secant degree is well defined even for the generic value, in the perfect case. The generic identifiability in a perfect case is rare, but when it happens it implies that we have a \emph{canonical form} (see e.g. \cite{Wa, Me}). Having  a canonical form means that the generic element $t\in \mathbb{P}^n=\langle X \rangle$ can be written in a unique way as a sum of $r$ elements on $X$ if $\sigma_r(X)$ is the first secant variety filling $\mathbb{P}^n$. The most celebrated  case when this situation appears is the famous Pentahedral Theorem of Sylvester: the generic quaternary cubic can be written in a unique way as a sum of 5 cubic forms.

Let $H$ be a general hyperplane section of $X$ tangent at $r$ general points $t_1, \ldots , t_{r}\in X$ with $r$ sub-generic (i.e. $\sigma_r(X) \subsetneq \langle X \rangle$), the \emph{contact locus} of $H$ is the union of the irreducible components of $\mathrm{Sing}(H)$ containing $t_1,\ldots  , t_r$. Remark that since $t_1, \ldots , t_r$ are general points, then the contact locus is equidimensional. Now $X$ is \emph{$r$-weakly defective}  if the general $(r+1)$-tangent hyperplane to $X$ has a contact locus of positive dimension (these concepts were introduced in \cite{cc}).\\
It is worth to remark that  finding a contact locus of positive dimension is not enough for claiming the non identifiability of the generic element (while the viceversa is true: if the contact locus is zero-dimensional then we have the uniqueness of the decomposition).  
Nevertheless there is a more refined notion that is more closely related to identifiability, namely the \emph{tangentially weakly defectiveness}.
Let $p_1, \ldots , p_r\in X$ be $r$ general points of a variety $X$; the $r$-tangentially contact locus of  $X$ is the set of points $\{p\in X \, | \, T_pX\subset \langle T_{p_1}X , \ldots , T_{p_r}X\rangle\}$.
A variety $X$ is said to be $r$-tangentially weakly defective if the $r$-tangentially contact
locus has positive dimension. If $X$ is not  $r$-tangentially weakly defective then we have the identifiability of the generic element of $\sigma_r(X) $(\cite[Proposition 2.4]{co}). This is not an ``~if and only if~" criterion, but still the $r$-tangentially contact locus of $X$ gives the right information on the number of decompositions of the generic element of $\sigma_r(X)$: in fact the $r$-secant degree of $X$ is equal to the $r$-secant degree of the $r$-tangentially contact locus of $X$ (cfr \cite{cc}).

\medskip 

In this paper we focus on the case of $X$ being a Grassmann variety {in its Pl\"uecker embedding} $Gr(\mathbb{P}^{k}, \mathbb{P}^{n})\subset \mathbb{P}(\bigwedge^{k+1}\mathbb{C}^{n+1})$. It parameterizes projective classes of skew-symmetric tensors that can be written as $v_1\wedge \cdots \wedge v_{k+1}$ with $v_i\in \mathbb{C}^n$ for $i=1, \ldots , k+1$. Therefore we will say that $t\in \mathbb{P}(\bigwedge^{k+1}\mathbb{C}^{n+1})$ has  skew-symmetric rank  $r$ if it belongs to a $\mathbb{P}^{r-1}$  which is $r$-secant  to 
 $Gr(\mathbb{P}^{k}, \mathbb{P}^{n})$, {with minimal $r$}. Since we will always deal with skew-symmetric tensors, there won't be any risk of confusion  if we will simply say that such a $t$ has \emph{rank} $r$.

\medskip

On defective secant varieties to Grassmann varieties there is an open conjecture {(stated independently in \cite{AOP,  BDG, CGG})} that says that defective Grassmannians  occur only for $Gr(\mathbb{P}^1, \mathbb{P}^n )$ for any $n$, $Gr(\mathbb{P}^2, \mathbb{P}^6 )$,  $Gr(\mathbb{P}^3, \mathbb{P}^7 )$, $Gr(\mathbb{P}^2, \mathbb{P}^8 )$ (see also \cite{Ada} for a recent proof for $\sigma_s(Gr(\mathbb{P}^k, \mathbb{P}^n ))$ with $s\leq 12$).

A classical result due to C. Segre (see \cite{Se}) shows that  $Gr(\mathbb{P}^2, \mathbb{P}^5)$ has  the 2-nd secant degree equal to 1, i.e. there is a canonical form for the generic element in   $\mathbb{P}(\bigwedge^3\mathbb{C}^6)$ that is therefore of type $[v_1\wedge v_2+ w_1\wedge w_2]$ with $v_i,w_i\in \mathbb{C}^6$, $i=1,2$.
After the example of C. Segre, the next interesting perfect cases are  $Gr(\mathbb{P}^3, \mathbb{P}^8 )$ and  $Gr(\mathbb{P}^4, \mathbb{P}^8 )$ (dual to each other) for which the secant degree is unknown. 
{In order to have a numerical evidence on the behavior of these two cases we  firstly made use of Bertini (\cite{Bertini}): it is possibile to show that the decompositions of the generic element in $\mathbb{P}(\bigwedge^4\mathbb{C}^9)$ as a sum of 6 elements in $Gr(\mathbb{P}^3, \mathbb{P}^8 )$ are a finite number. The  number of decompositions that we found with Bertini is  high (more than 7000). 
Bertini software is a good tool to have a numerical  evidence on the order of magnitude of the number of the decompositions, but we did not pursue this path since, having found such a big number of decompositions, we won't ever discover the precise amount of them by only using Bertini (see \cite{hoos} for a first application of homotopy continuation method with Bertini to the study of tensors identifiability, and \cite{bdhm} for its application to a new numerical algorithm for tensor decomposition). What we can claim is the following:  since we are in a perfect case, the fact that we found at least two numerically different decompositions, implies that  the generic element of $\bigwedge^4\mathbb{C}^9$ is not 6-identifiable. In fact in any perfect case the map from the abstract  $r$-secant variety $S_r=\{(x_1, \ldots , x_r;t)\in X^{r}\times \mathbb{P}(\langle X \rangle)\, | \, t\in \langle x_1, \ldots, x_r \rangle\}$ to the $r$-secant variety itself is generically finite, therefore, by Zariski's main theorem (see \cite{Za}), if the map was birational it would have connected fibers, but one could check, by computing the dimension of the tangent space, 
that at least one of the two different decompositions  is an isolated point of the fiber. As already anticipated, we  included here these considerations for sake of completeness but won't work out this argument within the manuscript.
}

\medskip

In this paper we firstly  compute the contact locus of all the highest secant varieties of the Grassmannians $Gr(\mathbb{P}^k, \mathbb{P}^n)$ that do not fill the ambient space for $n+1 \leq 14$. 
Secondly we find that, among the non-defective ones,  the only ones having positive dimensional  contact locus are 
$\sigma_3(Gr(\mathbb{P}^2, \mathbb{P}^7))$ and $\sigma_5(Gr(\mathbb{P}^2, \mathbb{P}^9))$. 
In the first case we find that the generic element  of $\sigma_3(Gr(\mathbb{P}^2, \mathbb{P}^7))$  is actually identifiable, therefore this is an example of a 3-weakly-defective Grassmannian having identifiable generic elements. An important remark in this respect will be Proposition \ref{duality} where we show that the dual variety of $\sigma_3(Gr(\mathbb{P}^2, \mathbb{P}^7))$ is $\sigma_2(Gr(\mathbb{P}^2, \mathbb{P}^7))$. {It will turn out that  $\sigma_2(Gr(\mathbb{P}^2, \mathbb{P}^7))$, $\sigma_3(Gr(\mathbb{P}^2, \mathbb{P}^7))$ and $\sigma_5(Gr(\mathbb{P}^2, \mathbb{P}^9))$ are the only weakly-defective secant varieties being not defective for $n<14$.}
The second case of $\sigma_5(Gr(\mathbb{P}^2, \mathbb{P}^9))$ is a new example for non-identifiability {and it is the unique one among the non-defective cases for $n<14$}.  In Proposition \ref{main} we show that the generic order 3  skew-symmetric tensor of $\mathbb{C}^{10}$ of rank 5 belongs to exactly two $\mathbb{P}^4$'s 5-secant to $Gr(\mathbb{P}^2, \mathbb{P}^9)$.

{Our main result is Theorem \ref{Main} where we compute all the secant degrees for any Grassmannian if $n<14$. Finally we conclude the paper with two Corollaries, \ref{corollweak} and \ref{corolltgweak}, where we classify all the weakly defective cases and all the tangentially weakly defective cases for the same range $n<14$.}

\section{New non-identifiable Grassmannian}\label{NewNonIdent}

In order to compute the contact locus for all the secant varieties of the Grassmannians $Gr(\mathbb{P}^{k}, \mathbb{P}^{n})$ that does not fill the ambient space for $n+1 \leq 14$ we use Macaulay2 \cite{m2} (see the file {\tt{grascontactlocus.m2}} in the ancillary material).
For  those computations we have used the Hessian criterion introduced in \cite{cov} (see \cite[Lemma 4.3, Lemma 4.4, and Theorem 4.5]{cov}) suitably adapted to skew-symmetric tensors. We stopped to $n+1=14$ because, after such a value of $n$, the computational cost of running the program becomes too high. {The main theorem of this paper is the following:}
\begin{theorem}\label{Main} $\,$
\begin{enumerate} 
\item\label{1} 
\begin{enumerate}
\item The Grassmannian $Gr(\mathbb{P}^2,\mathbb{P}^7)$ is {2 and 3-weakly defective and the generic elements of $\sigma_2(Gr(\mathbb{P}^2,\mathbb{P}^7))$ and $\sigma_3(Gr(\mathbb{P}^2,\mathbb{P}^7))$ are identifiable}.
\item The dual variety $\sigma_3(Gr(\mathbb{P}^2,\mathbb{P}^7))^{\vee}$ is $\sigma_2(Gr(\mathbb{P}^2,\mathbb{P}^7))$. 
\end{enumerate}
\item\label{3} The Grassmannian  $Gr(\mathbb{P}^k,\mathbb{P}^n)$ is $r$-identifiable for $n<14$ and $r$ sub-generic except for:
\begin{enumerate}
{\item $\sigma_r(Gr(\mathbb{P}^1, \mathbb{P}^n))$, $2r \leq n+1$;}
\item $\sigma_3(Gr(\mathbb{P}^2, \mathbb{P}^6))\simeq \sigma_3(Gr(\mathbb{P}^3,\mathbb{P}^6))$;
\item\label{due} $\sigma_5(Gr(\mathbb{P}^2, \mathbb{P}^9)) \simeq \sigma_5(Gr(\mathbb{P}^6, \mathbb{P}^9))$;
\item $\sigma_3(Gr(\mathbb{P}^3, \mathbb{P}^7))$;
\item  $\sigma_4(Gr(\mathbb{P}^3, \mathbb{P}^7))$;
\item $\sigma_4(Gr(\mathbb{P}^2, \mathbb{P}^8)) \simeq \sigma_4(Gr(\mathbb{P}^5, \mathbb{P}^8))$.
\end{enumerate}
Moreover the $5^{\mathrm{th}}$-secant degree of $Gr(\mathbb{P}^2,\mathbb{P}^9)$ is 2 (case (\ref{due})), in all the other exceptional cases
 the corresponding $r^{\mathrm{th}}$-secant degree of $Gr(\mathbb{P}^k, \mathbb{P}^n)$ is infinity.
\end{enumerate}
\end{theorem}

\begin{proof}
Item (\ref{1}) is proved in  Section \ref{sub1}. 
Item (\ref{due}) is proved in Section \ref{sub2}. All the other cases listed above correspond to defective secant varieties (cfr. \cite{Ada,AOP,  BDG, CGG}). 

The fact that there are no other exceptions is a consequence of the fact that there are no other positive dimensional contact loci except for $\sigma_3(Gr(\mathbb{P}^2, \mathbb{P}^7))$ and $\sigma_5(Gr(\mathbb{P}^2, \mathbb{P}^9))$ among the non defective cases: {clearly if $X$ is an $r$-weakly defective variety then it is also $(r+k)$-weakly defective for any $1\leq k < \min \{s\in \mathbb{N} \; | \; \sigma_s(X)=\langle X \rangle\}-r$; and if $X$ is $r$-identifiable then it is also $(r-k)$-identifiable for any $0\leq k\leq r-1$. Since for $Gr(\mathbb{P}^2,\mathbb{P}^9)$ we have proved by direct computation that it is not 4-weakly defective, hence its generic element is  4-identifiable}. 

{Finally the 2-identifiability of $Gr(\mathbb{P}^2,\mathbb{P}^6)$ and $Gr(\mathbb{P}^3,\mathbb{P}^7)$ and the 3-identifiability of $Gr(\mathbb{P}^2,\mathbb{P}^8)$  were directly computed with Macaulay 2. More precisely we found a 6 dimensional contact locus for $\sigma_2(Gr(\mathbb{P}^2,\mathbb{P}^6))$, so it is potentially weakly defective, but we computed that $Gr(\mathbb{P}^2,\mathbb{P}^6)$ is not 2-tangentially weakly defective, therefore we have the 2-identifiability for its generic element, while $Gr(\mathbb{P}^3,\mathbb{P}^7)$ is not 2-weakly defective and $Gr(\mathbb{P}^2,\mathbb{P}^8)$ is not 3-weakly defective.} 
\end{proof}

\subsection{Identifiability for the generic element of $\sigma_3(Gr(\mathbb{P}^2, \mathbb{P}^7))$}\label{sub1}

The computation that we have done with Macaulay2 \cite{m2} (see {\tt{grascontactlocus.m2}} in the ancillary material) shows that  $\sigma_3(Gr(\mathbb{P}^2, \mathbb{P}^7))$ has a positive dimensional contact locus, i.e. that it is weakly-defective, with ``~high probability~".  
Before investigating on the identifiability of the generic element we would like to show that $Gr(\mathbb{P}^2, \mathbb{P}^7)$ is indeed $3$-weakly defective. We will make use of the fact that a variety $X$ is  $r$-weakly defective if and only if the dimension of the dual variety to $\sigma_r(X)$ is smaller than $\dim(\mathbb{P}\langle X \rangle)-r$ (see \cite{cc}). We will also say that a variety $X$ is \emph{dual defective} if its dual variety $X^{\vee}$ is not a hypersurface.

\begin{proposition}\label{duality} The dual variety $(\sigma_3(Gr(\mathbb{P}^2, \mathbb{P}^7)))^{\vee}$ is $\sigma_2(Gr(\mathbb{P}^2, \mathbb{P}^7))$ and the Grassmannian $Gr(\mathbb{P}^2, \mathbb{P}^7)$ is 3-weakly defective with a 7 dimensional contact locus.
\end{proposition}

\begin{proof} Remark that $SL(8)$ has only a finite number of orbits on $\mathbb{P}(\bigwedge^3\mathbb{C}^8)$. G.B. Gurevich in \cite[VII, \S 35.4]{gur} gave the complete  classification of  those orbits; their dimensions are computed by  D. $\check{\mathrm{Z}}$. Djokovi\'c in \cite[Table I]{drag}. We retrieve this classification in  our Table \ref{gur}. 

\begin{notation}[for Table \ref{gur}]\label{notgur}
The table is splitted  vertically in two parts: on the same row we write the orbits that are dual to each other. We have checked them via dimension count: since $SL(8)$ has only a finite number of orbits on $\mathbb{P}(\bigwedge^3\mathbb{C}^8)$, then the dual variety of an orbit closure remains a homogeneous variety, therefore it has to be one of those classified by Gurevich in \cite[VII, \S 35.4]{gur}. The only ambiguity exists for XV and XIX; we prove this case along the present proof. We follow the notation of \cite{gur}: in the first ($5^{th}$ resp.) column the numbers of the orbits are the same used by Gurevich in \cite[VII, \S 35.4]{gur}; in the second ($7^{th}$ resp.) column we write the canonical form  (C.F.) of an element in each orbit; in the third ($8^{ve}$ resp.) column we write the affine dimension (D.) of the corresponding orbit and in the $4^{th}$ (last resp.) column we write the variety of the orbit closure. The  notation for the canonical form used in Table \ref{gur} is the following $[abc][qrs]:=a\wedge b \wedge c+ q \wedge r\wedge s$ where $a,b,c,q,r,s\in \mathbb{C}^8$. Moreover, in that table ``~$G$~" stays for $Gr(\mathbb{P}^2, \mathbb{P}^7)$; ``~C~" for restricted chordal variety; ``~$\tau$~" for the tangential variety to $Gr(\mathbb{P}^2, \mathbb{P}^7)$; ``~$\sigma_i$~" for $\sigma_i(Gr(\mathbb{P}^2, \mathbb{P}^7))$, $i=2,3$; ``~$J(G,X)$~" for the join variety among  $Gr(\mathbb{P}^2, \mathbb{P}^7)$ and the variety $X$; and ``~$S_i$~" for the subspace variety $Sub_i(\bigwedge^3\mathbb{C}^8):=\{t\in \mathbb{P}(\bigwedge^3\mathbb{C}^8) \, | \,  \exists \, \mathbb{C}^i \subset \mathbb{C}^8 \text{ s.t. }  t\in \mathbb{P}(\bigwedge^3\mathbb{C}^i)\}$, $i=6,7$.  We refer to \cite{gur1} for the complete classification of all other orbits.
\end{notation}

{\footnotesize{
\begin{table}[!h]
\centering
\caption{Classification of the orbits of $SL(8)$  on $\mathbb{P}(\bigwedge^3\mathbb{C}^8)$.  Notation is settled in Notation \ref{notgur}.}
\label{gur}
\begin{tabular}{llll | llll}
 & C.F.  & D.  & Var. &  & C.F. & D.& Var.  \\\hline \hline
 I& $w=0$ & 0  & &XXIII&$[abc][qrs][aqp][brp]$&56 & $\mathbb{P}^{55}$ \\
 &&&&&$[csp][bst][crt]$&&\\
  \hline 
 II& $[qrs]$ & 16 & $G$ & XXII &$[abc][qrs][aqp][brp]$&55&$G^{\vee}$\\
 &&&&&$[bst][crt]$&&\\
 \hline 
   III& $[aqp][brp]$ &25  &C  &  XXI &$[abc][qrs][aqp][bst]$&53&C$^{\vee}$\\ \hline 
 IV& $[aqr][brp][cpq]$ & 31  &  $\tau$& XX &$[qrs][aqp][brp][csp]$&52&$\tau^{\vee}$\\
 &&&&&$[bst][crt]$&&\\ \hline
 V& $[abc][pqr]$ & 32  &$\sigma_2=S_6$  & XIX &$[aqp][brp][csp][bst]$&48 & $\sigma_3$ \\
 &&&&&$[crt]$&&\\
  \hline 
 VI& $[aqp][brp][csp]$ & 28  &  $S_7^{\vee}$&X&  $[abc][qrs][aqp][brp]$&  42&  $S_7$\\ 
&&&&&   $[csp]$&& $$\\ \hline
 VII&  $[qrs][aqp][brp][csp]$
 & 35 & & XVIII &$[qrs][aqp][brp][bst]$&50&\\
 & &  &  & &$[crt]$&&\\ \hline
 VIII&  $[abc][qrs][aqp]$ & 38 & {$J(G, \tau)^{\vee}$} & XVII &$[aqp][brp][bst][crt]$&47&{$J(G, \tau)$}\\ \hline
 IX&$[abc][qrs][aqp][brp]$  & 41   &$J(G,C)^\vee$  &XVI  &$[aqp][bst][crt]$& 41&$J(G,C)$\\ \hline
 XI& $[aqp][brp][csp][crt]$ & 40 &   &XV &$[abc][qrs][aqp][brp]$ &48&\\
&  &  &  & & $[csp][crt]$&&\\ \hline
  XII& $[qrs][aqp][brp][csp]$ &43  & & XIV&$[abc][qrs][aqp][brp]$ &46&\\
 & $[crt]$ &  &  & &$[crt]$ &&\\ \hline
 XIII&  $[abc][qrs][aqp][crt]$&44  &  &&SELF DUAL&& \\\hline
\end{tabular}
\end{table}
}}

For sake of completeness we include in Table \ref{cont} the containment diagram of the orbit closures of  $SL(8)$ on $\mathbb{P}(\bigwedge^3\mathbb{C}^8)$ (we want to thank W.A. de Graaf for his help with  SLA \cite{SLA} GAP4 \cite{GAP} package, in drawing this diagram; anyway the same diagram is also described in detail in \cite[Figure 1]{drag}).

{\footnotesize{
\begin{table}[h!]
\caption{Containment diagram for the  orbit  closures of  $SL(8)$ on $\mathbb{P}(\bigwedge^3\mathbb{C}^8)$ together with their affine dimensions.}
\label{cont}$$
\xy
(15,0)*+{I}="o1";
(15,8)*+{II}="o2";
(15,16)*+{III}="o3";
(5,24)*+{VI}="o6";
(15,32)*+{IV}="o4";
(25,40)*+{V}="o5";
(5,48)*+{VII}="o7";
(15,56)*+{VIII}="o8";
(25,64)*+{XI}="o11";
(15,72)*+{IX}="o9";
(35,72)*+{XVI}="o16";
(5,80)*+{X}="o10";
(15,88)*+{XII}="o12";
(15,96)*+{XIII}="o13";
(15,104)*+{XIV}="o14";
(15,112)*+{XVII}="o17";
(5,120)*+{XV}="o15";
(25,120)*+{XIX}="o19";
(5,128)*+{XVIII}="o18";
(15,136)*+{XX}="o20";
(15,144)*+{XXI}="o21";
(15,152)*+{XXII}="o22";
(15,160)*+{XXIII}="o23";
(-15,0)*{0};
(-15,8)*{16};
(-15,16)*{25};
(-15,24)*{28};
(-15,32)*{31};
(-15,40)*{32};
(-15,48)*{35};
(-15,56)*{38};
(-15,64)*{40};
(-15,72)*{41};
(-15,80)*{42};
(-15,88)*{43};
(-15,96)*{44};
(-15,104)*{46};
(-15,112)*{47};
(-15,120)*{48};
(-15,128)*{50};
(-15,136)*{52};
(-15,144)*{53};
(-15,152)*{55};
(-15,160)*{56};
{\ar@{-} "o1"; "o2"};
{\ar@{-} "o2"; "o3"};
{\ar@{-} "o6"; "o3"};
{\ar@{-} "o4"; "o3"};
{\ar@{-} "o4"; "o5"};
{\ar@{-} "o6"; "o7"};
{\ar@{-} "o4"; "o7"};
{\ar@{-} "o5"; "o8"};
{\ar@{-} "o7"; "o8"};
{\ar@{-} "o8"; "o11"};
{\ar@{-} "o8"; "o9"};
{\ar@{-} "o11"; "o16"};
{\ar@{-} "o9"; "o10"};
{\ar@{-} "o9"; "o12"};
{\ar@{-} "o11"; "o12"};
{\ar@{-} "o12"; "o13"};
{\ar@{-} "o13"; "o14"};
{\ar@{-} "o14"; "o17"};
{\ar@{-} "o16"; "o17"};
{\ar@{-} "o10"; "o15"};
{\ar@{-} "o14"; "o15"};
{\ar@{-} "o17"; "o19"};
{\ar@{-} "o15"; "o18"};
{\ar@{-} "o17"; "o18"};
{\ar@{-} "o18"; "o20"};
{\ar@{-} "o19"; "o20"};
{\ar@{-} "o20"; "o21"};
{\ar@{-} "o21"; "o22"};
{\ar@{-} "o22"; "o23"};
\endxy
$$
\end{table}}}
\smallskip 
The variety $\sigma_3(Gr(\mathbb{P}^2,\mathbb{P}^7))$ is not defective (from the dimension of the secant variety point of view), therefore its affine cone has dimension 48. Gurevich in \cite{gur} shows that $SL(8)$ generates two orbits of affine dimension 48 in $\mathbb{P}(\bigwedge^3\mathbb{C}^8)$: XV and XIX (as illustrated in Table \ref{cont}). One of them must be the open part of $\sigma_3(Gr(\mathbb{P}^2,\mathbb{P}^7))$. Gurevich also shows that the dual variety of the closure of XIX has affine dimension 32 and its open part is the orbit of $a\wedge b \wedge c+ p\wedge q\wedge r$ (it is represented by V in Table \ref{gur}), i.e. the closure of V is obviously $\sigma_2(Gr(\mathbb{P}^2, \mathbb{P}^7))$. Therefore if we prove that  $\sigma_3(Gr(\mathbb{P}^2,\mathbb{P}^7))$  is the closure of XIX we are done.

Either if $\sigma_3(Gr(\mathbb{P}^2,\mathbb{P}^7))$ is the closure of XIX or of XV, it is dual defective: in one case its dual variety would have affine dimension 32 and in the other 40 (in both cases the dual variety of  $\sigma_3(Gr(\mathbb{P}^2,\mathbb{P}^7))$ won't be {of dimension 55-3=52}). Now the point is that there is a link  between the contact locus of a secant variety and its dual variety (as it is shown in \cite{cc}). More precisely:
the codimension of the dual variety of a secant variety $\sigma_k(X)$ which is not defective but with contact locus of projective dimension $c$  is 
\begin{equation}\label{clcc}
\mathrm{codim}(\sigma_k(X)^\vee)=k(c+1).
\end{equation}
This leads us to the following two possibilities: \begin{itemize}
\item if $\sigma_3(Gr(\mathbb{P}^2,\mathbb{P}^7))$ was the closure of XV then its dual variety would have codimension $3(c+1)=56-40=16$,
but this is impossible because $c$ has to be a natural number;
\item If $\sigma_3(Gr(\mathbb{P}^2,\mathbb{P}^7))$ is the closure of XIX then its dual variety  has codimension $3(c+1)=56-32=24$, this is clearly possible and it is the only possibility left. 
\end{itemize}
This shows that $\sigma_3(Gr(\mathbb{P}^2,\mathbb{P}^7))^{\vee}=\sigma_2(Gr(\mathbb{P}^2,\mathbb{P}^7))$.
Remark that this also shows that $Gr(\mathbb{P}^2,\mathbb{P}^7)$ is {2 and 3-weakly-defective and the dimensions of the contact loci  of $\sigma_2 (Gr(\mathbb{P}^2,\mathbb{P}^7))$ and $\sigma_3(Gr(\mathbb{P}^2,\mathbb{P}^7))$
 are $3$ and $7$ respectively}. \end{proof}
 
 \begin{question} It would be newsworthy to give a geometric description of the duality $\sigma_3(Gr(\mathbb{P}^2,\mathbb{P}^7))^{\vee}=\sigma_2(Gr(\mathbb{P}^2,\mathbb{P}^7))$. An interesting fact for this purpose is that $\sigma_2(G(\mathbb{P}^2, \mathbb{P}^7))$ is equal to the so called subspace variety $Sub_6(\bigwedge^3\mathbb{C}^8):=\{t\in \mathbb{P}(\bigwedge^3\mathbb{C}^8) \, | \,  \exists \, \mathbb{C}^6 \subset \mathbb{C}^8 \text{ s.t. }  t\in \mathbb{P}(\bigwedge^3\mathbb{C}^6)\}$ (crf \cite[Ex. 7.1.4.3]{jm}). One containment is obvious and it holds for any secant variety of any Grassmannian with the correct adjusting of indices, the other containment is a peculiarity of this specific case. 
 \end{question}
 
 \begin{remark} As already remarked, the projective duality in Table \ref{gur} is performed via computation of the dimensions of the dual varieties of the orbit closure of any generator, and via a specific argument for XV and XIX  given in the proof of Proposition \ref{duality}. It is worth to remark that this duality almost corresponds to the duality of the arithmetic characters showed in \cite{gur}: they agree in almost all cases except for VI and X that are projectively dual  to each other according to our computations, while Gurevich in \cite[VII, \S 35.4]{gur} explicitly writes that  those two orbits don't have any dual orbit. It is a very interesting and peculiar phenomenon that the projective duality does not correspond to the duality of arithmetic characters.
\end{remark}

\begin{corollary}\label{repeated} The variety $Gr(\mathbb{P}^2,\mathbb{P}^7)$ is {2 and 3-weakly defective}.
\end{corollary}

\begin{proof} The duality $\sigma_3(Gr(\mathbb{P}^2,\mathbb{P}^7))^{\vee}=\sigma_2(Gr(\mathbb{P}^2,\mathbb{P}^7))$ together with the formula (\ref{clcc}) (cfr. \cite{cc}) show that the contact locus of $\sigma_2(Gr(\mathbb{P}^2,\mathbb{P}^7))$ has dimension 3 {and the contact locus of $\sigma_3(Gr(\mathbb{P}^2,\mathbb{P}^7))$ has dimension 7}.
\end{proof}

As already recalled in the introduction, the weakly defectiveness is not sufficient to claim anything about the identifiability.

\begin{proposition}\label{3id} The Grassmannian $Gr(\mathbb{P}^2,\mathbb{P}^7)$ is 3-identifiable.
\end{proposition}

\begin{proof} 
We computed with Macaulay2 (\cite{m2}) the tangentially contact locus $\mathcal{T}$ at three points of $Gr(\mathbb{P}^2,\mathbb{P}^7)$;
it turns out to be the union of three disjoint $\mathbb{P}^3$'s,  each one passing through one and only one of the tangent points, and a $\mathbb{P}^5$ not passing to any one of the three points of tangency. 

More precisely, the three points that we chose (before the Pl\"uecker embedding) were the following:
 $$q_1=\left( \begin{array}{cccccccc}
1&0&0&0&0&0&0&0\\
0&1&0&0&0&0&0&0\\
0&0&1&0&0&0&0&0\\
\end{array}\right), \; q_2=
\left( \begin{array}{cccccccc}
1&0&0&1&0&0&0&0\\
0&1&0&0&1&0&0&0\\
0&0&1&0&0&1&0&0\\
\end{array}\right),$$
$$q_3=
\left( \begin{array}{cccccccc}
1&0&0&0&0&1&0&0\\
0&1&0&0&0&0&1&0\\
0&0&1&0&0&0&0&1\\
\end{array}\right)$$
We computed the 3-tangentially contact locus of the span of the three tangent spaces at these points. We found that  in the coordinates $\{a_0,a_{1,1}, \ldots , a_{3,5}\}$ of the $\mathbb{P}^{15}$ parameterizing $Gr(\mathbb{P}^2,\mathbb{P}^7)$ the ideal of the tangentially contact locus is the intersection of the following four ideals: 
{\small{
$$I(\Pi_1)=(a_{3,4},a_{1,4},a_{2,4},a_{3,1},a_{1,1},a_{2,1},a_{3,2},a_{1,2},a_{
      2,2},a_{3,5},a_{1,5},a_{2,5}),$$
$$I(\Pi_2)=(a_{3,4},a_{1,4},a_{2,4}-1,a_{3,1},a_{1,1},a_{2,1},a_{3,2},a_{1,2},a
      _{2,2},a_{3,3}+a_{3,5}-1,a_{1,3}+a_{1,5}-1,a_{2,3}+a_{2,5}),
$$
$$I(\Pi_3)=(a_{3,4},a_{1,4},a_{2,4},a_{2,1}+a_{2,3},a_{3,1}+a_{3,3}-1,a_{1,1}+a
      _{1,3}-1,a_{3,2},a_{1,2},a_{2,2}-1,a_{3,5},a_{1,5},a_{2,5}),
$$
$$I(\Pi_4)=(a_{3,4},a_{1,4},a_{3,1},a_{1,1},a_{3,3}-1,a_{1,3}-1,a_{3,2},a_{1,2
      },a_{3,5},a_{1,5}).
$$
}}
Clearly all the $\Pi_i$'s are linear and it is very easy to check that they remain linear even after the Pl\"uecker embedding $p_{3,8}:\mathbb{P}^{15}\rightarrow \mathbb{P}^{55}$. Moreover $\Pi_i\simeq p_{3,8}(\Pi_i)\simeq \mathbb{P}^3$ for $i=1,2,3$ and $\Pi_4\simeq p_{3,8}(\Pi_4)\simeq \mathbb{P}^5$. It's again an easy check that $q_i\in \Pi_i$ for $i=1,2,3$ and that $q_i\notin\Pi_j$ for $i\neq j$, $i=1,2,3$ and $j=1,2,3,4$. Remark also that the three $\mathbb{P}^3$'s have no common components. Now it's clear that  the generic point on a honest 3-secant plane to $\mathcal{T}$ can be written in a unique way as a linear combination of 3 points of $\mathcal{T}$. As already recalled in the Introduction this 
suffices to claim the $3$-identifiability of $Gr(\mathbb{P}^2,\mathbb{P}^7)$ (cfr \cite{cc}).
\end{proof}

{\begin{corollary}The Grassmannian $Gr(\mathbb{P}^2,\mathbb{P}^7)$ is 2-identifiable.
\begin{proof} By definition of $r$-identifiability if $X$ is $r$-identifiable then it is also $(r-k)$-identifiable for any $0\leq k\leq r-1$.
\end{proof}
\end{corollary}}

\begin{remark} We would like to point out a very peculiar phenomenon that we have not found before in the literature. In the computation of the 3-tangentially contact locus of  $Gr(\mathbb{P}^2,\mathbb{P}^7)$ (in the proof of Proposition \ref{3id}) we found four components: three of them pass through the points of tangency, while the other one doesn't pass through any one of them.
\end{remark}

\begin{remark} The fact that $\sigma_2(Gr(\mathbb{P}^2, \mathbb{P}^7))$ and $\sigma_3(Gr(\mathbb{P}^2, \mathbb{P}^7))$ are weakly defective but their generic element is identifiable is not a new phenomenon: in  \cite[Lemma 2.5 and Theorem 2.7]{cov} is shown an analogous example in a case of secant varieties to Segre varieties. 

\end{remark}

\subsection{Two decompositions for the generic element of $\sigma_5(Gr(\mathbb{P}^2,\mathbb{P}^9))$}\label{sub2}

We have computed with Macaulay2 \cite{m2} that  $\sigma_5(Gr(\mathbb{P}^2,\mathbb{P}^9))$ has a positive dimensional contact locus with ``~very high probability~", i.e. that it should be weakly-defective. Here we want to prove that this is actually the case and moreover we can also show that its generic element is not identifiable.
{More precisely, in Corollary \ref{main} we will show  that the generic element of $\sigma_5(Gr(\mathbb{P}^2, \mathbb{P}^9))$ has exactly 2 decompositions as a sum of 5 points in $Gr(\mathbb{P}^2,\mathbb{P}^9)$.}

\medskip 

Firstly we like to recall what a \emph{3-torsion scroll} of $\mathbb{P}^2$'s in $\mathbb{P}^9$ is. Fix an origin on an elliptic normal curve $E\subset \mathbb{P}^9$ and a 3-torsion point $P$ on $E$.
Then for each point $Q$ in $E$, the three points $Q$, $P+Q$ and $2P+Q$ span a plane.  As $Q$ moves, these planes form a scroll, the so called a 3-torsion scroll. It is a special case of a \emph{3-translation scroll} defined analogously without requiring that $P$ is necessarily a 3-torsion point.  
The 3-torsion scroll of $\mathbb{P}^2$'s in $\mathbb{P}^9$ has degree 10,  in fact the general 3-translation scroll have degree 30, hence our 3-torsion scroll has degree 10 because the 3 secant planes 
$\mathbb{P}(\langle Q,Q+P,Q+2P \rangle),
\mathbb{P}(\langle Q-P,Q,Q+P\rangle)$ and 
$\mathbb{P}(\langle Q-2P,Q-P,Q\rangle )$
coincide. Therefore the corresponding curve in $Gr(\mathbb{P}^2,\mathbb{P}^9)$ has degree 10 because for a scroll over a curve the degree as a scroll coincides with the Pl\"uecker degree of the curve in the Grassmannian. Moreover if we look at the 3-torsion scroll as a rank-3 vector bundle  it is indecomposable, and, viceversa, if a vector bundle ${\mathcal{E}}$ over an elliptic normal curve of degree 10 is indecomposable then $\mathbb{P}(\mathcal{E})$ is the 3-torsion scroll (see. \cite[Lemma 3 and Remark 31]{adhpr} and  \cite{CaCi}); the  fact that it is indecomposable is a consequence of the fact that such a scroll is the symmetric product of 3 times the elliptic curve  cf. \cite{a1,a, Se2};  see also \cite{ch} for the analogous description for 2-torsion scrolls).

\begin{proposition} The contact locus of $\sigma_5(Gr(\mathbb{P}^2,\mathbb{P}^9))$ is the 3-torsion scroll of $\mathbb{P}^2$'s in $\mathbb{P}^9$.
\end{proposition}

\begin{proof} 
We computed with Macaulay2 (\cite{m2})  that through 5 specific random points of $Gr(\mathbb{P}^{2},\mathbb{P}^{9})$ there is only one irreducible curve in the contact locus and it is an elliptic curve (see the ancillary material to the arXiv version of the present paper). The curve that we have found with Macaulay2 is birational to our contact locus since we have done the computation on an affine chart, this allows to say that since the one we found is an elliptic curve then the one in the contact locus remains an elliptic curve. Unfortunately the degrees of the two curves may not be the same. The direct computation of the degree of the curve was impossible with Macaulay2 and too long with Bertini (\cite{Bertini}) for which we used the techniques of pseudo Witness sets developed  in \cite{bdhm}. We therefore computed the associated scroll and we found out that its degree is 10. This is sufficient to claim that the contact locus is an elliptic curve of degree 10 since for a scroll over a curve its degree as a scroll coincides with the degree of the corresponding curve in the Grassmannian. Moreover, since by  \cite[Theorem 2.4]{cc2} the contact locus of $\sigma_5(Gr(\mathbb{P}^2, \mathbb{P}^9))$ spans a $\mathbb{P}^9$ then our elliptic curve of degree 10 is also normal.

Now, since we have shown that though 5 random point there is a unique elliptic normal curve of degree 10, this is the situation for 5 general points by semicontinuity and this is sufficient to say that the concat locus of $\sigma_5(Gr(\mathbb{P^2}, \mathbb{P}^9))$ is exactly an elliptic normal curve of degree 10.

Now we want to prove that through 5 generic points of $Gr(\mathbb{P}^2,\mathbb{P}^9)$ there there is always a 3-torsion scroll. The Hilbert scheme of the elliptic normal curves of degree $d$ in $\mathbb{P}^{d-1}$ has dimension $d^2$. The conditions imposed by plane of $\mathbb{P}^9$ to be 3-secant to an elliptic normal curve of degree 10 in the 3 points $P-Q=Q-R=R-P$ are exactly 20. Now
 if we consider five 3-secant planes in $\mathbb{P}^9$ with this property, they impose   
$20\cdot  5=100=d^2$ conditions to the elliptic curves of degree 10, therefore we expect a finite number of elliptic curves with the property above.
Moreover the 3-torsion scroll is always contained in the Grassmannian by construction and we have shown that an elliptic curve $C$ of degree 10  is contained in the contact locus of 5 points.
In order to conclude it is sufficient to remind that the 3-torsion scroll corresponds to an indecomposable rank 3 vector bundle over an elliptic curve, moreover if  the vector bundle is  indecomposable, then its projectivization is the 3-torsion translation scroll (cf. \cite{adhpr, CaCi}).

Summing up: We have 5 specific points through which there is only one elliptic normal curve in the contact locus (this is the computation that we have done with Macaulay2 (\cite{m2}));
through 5 general points there is always a 3-torsion scroll that is a degree 10 elliptic normal curve which it is contained in the contact locus, it is irreducible and it spans a $\mathbb{P}^9$; therefore, by semicontinuity, we can say that the contact locus is given by only one elliptic normal curve of degree 10 spanning a $\mathbb{P}^9$ that is the 3-torsion scroll of $\mathbb{P}^2$'s in $\mathbb{P}^9$.
\end{proof}

\begin{corollary}\label{main} The generic element of $\sigma_5(Gr(\mathbb{P}^2,\mathbb{P}^9))$ has exactly 2 decompositions as a sum of 5 points in $Gr(\mathbb{P}^2,\mathbb{P}^9)$.
\end{corollary}

\begin{proof}  
The previous proposition shows that  the contact locus of $\sigma_5(Gr(\mathbb{P}^2,\mathbb{P}^9))$ is the 3-torsion scroll of $\mathbb{P}^2$'s of $\mathbb{P}^9$  which is a degree 10 elliptic normal curve in the Grassmannian . This is enough to conclude, in fact
having an elliptic {normal} curve as a contact locus leads to exactly two decompositions for the generic element of  $\sigma_5(Gr(\mathbb{P}^2,\mathbb{P}^9))$. In order to see this last fact it would be sufficient to quote \cite{cc}: the same argument on the equality between the $r$-th secant degree of the tangentially contact locus and the number of decomposition of the generic element in the $r$-th secant variety holds. Anyway, for the present specific example  this can be  shown geometrically. 
Fix 5 points on $\mathcal{C}$ and take all the $(\mathbb{P}^8)$'s containing them; they define a linear series and they intersect $\mathcal{C}$ in other 5 points (and no more). Moreover the two ($\mathbb{P}^4$)'s spanned by those two quintuple of points must intersect each other since they live in the same $\mathbb{P}^8$. This is again sufficient to conclude that we have exactly two decompositions for the generic element of $\sigma_5(Gr(\mathbb{P}^2,\mathbb{P}^9))$.
\end{proof}

{\begin{corollary}\label{corollweak} The non $r$-defective  Grassmannians $Gr(\mathbb{P}^k,\mathbb{P}^n)$ for $n<14$ are all non $r$-weakly defective except for:
\begin{enumerate}[(a)]
\item\label{vecchiabis} $r=2,3$  and $Gr(\mathbb{P}^2, \mathbb{P}^7)\simeq Gr(\mathbb{P}^{4}, \mathbb{P}^7)$, where the contact loci have dimensions 3 and 7 respectively;
\item\label{vecchia} $r=5$ and $Gr(\mathbb{P}^2,\mathbb{P}^9)\simeq Gr (\mathbb{P}^6, \mathbb{P}^9)$, where the contact locus has dimension 1;
\item\label{nuova} $r=2$ and $Gr(\mathbb{P}^2, \mathbb{P}^6)\simeq Gr(\mathbb{P}^{3}, \mathbb{P}^6)$, where the contact locus has dimension $6$.
\end{enumerate}
\end{corollary}
\begin{proof} Case (\ref{vecchiabis}) is Corollary \ref{repeated}. The dimensions of the contact loci  are computed in the proof of Proposition \ref{duality}  when we show that $\sigma_2(Gr(\mathbb{P}^2, \mathbb{P}^7))$ and  $\sigma_3(Gr(\mathbb{P}^2, \mathbb{P}^7))$ are dual to each other. 

In the proof of Proposition \ref{main} we showed that $Gr(\mathbb{P}^2,\mathbb{P}^9)$ is 5-tangentially weakly defective, therefore it is also 5-weakly defective. In the same proof we also showed that the contact locus is an elliptic normal curve. This proves case (\ref{vecchia}). As already said in the proof of Theorem \ref{Main}, the fact that $Gr(\mathbb{P}^2,\mathbb{P}^9)$ is not 4-weakly defective is done by direct computation.

The only case that we have not proved yet is (\ref{nuova}). We computed, with Macaulay2, the dimension of $(\sigma_2(Gr(\mathbb{P}^2, \mathbb{P}^6)))^{\vee}$, by considering $\sigma_2(Gr(\mathbb{P}^2, \mathbb{P}^6))$ to be the orbit of $e_0\wedge e_1 \wedge e_2 + e_3\wedge e_4 \wedge e_5$ via the action of $SL(7)$ in $\bigwedge^3\mathbb{C}^7$. It turns out that $\dim  (\sigma_2(Gr(\mathbb{P}^2, \mathbb{P}^6)))^{\vee}=21$, therefore,  by the displayed formula (\ref{clcc}) above, the contact locus has dimension 6.

The fact that all the regular cases (i.e. Grassmannians  with $r$-secant varieties of the expected dimension) not listed above are not weakly defective is a consequence of the computation that we have done in the proof of Theorem \ref{Main} that shows that in those cases all contact loci are 0-dimensional.
\end{proof}

\begin{corollary}\label{corolltgweak}The non $r$-defective  Grassmannians $Gr(\mathbb{P}^k,\mathbb{P}^n)$ for $n<14$ are all non $r$-tangentially weakly defective except for: 
\begin{enumerate}
\item $r=3$  and $Gr(\mathbb{P}^2, \mathbb{P}^7)$, where the tangentially contact locus has dimension  5;
\item $r=5$ and $Gr(\mathbb{P}^2,\mathbb{P}^9)$, where the tangentially contact locus has dimension 1.
\end{enumerate}

\end{corollary}
\begin{proof} Since the $r$-tangentially weakly defectiveness implies the $r$-weakly defectiveness, we have to check only weakly defective cases listed in Corollary   \ref{corollweak}.

We  computed with Macaulay2 that the $2$-tangentially contact locus of $Gr(\mathbb{P}^2, \mathbb{P}^7)$ is 0-dimensional. This suffices to prove that $Gr(\mathbb{P}^2, \mathbb{P}^7)$ is not 2-tangentially weakly defective.

In Proposition \ref{3id} we computed the 3-tangentially contact locus of $Gr(\mathbb{P}^2, \mathbb{P}^7)$ and we found that it is the union of three $\mathbb{P}^3$'s and a $\mathbb{P}^5$.

 In Proposition \ref{main} we showed that the 5-th secant degree of $Gr(\mathbb{P}^2, \mathbb{P}^9)$ is two, therefore we don't have the identifiability for the generic element of $\sigma_2(Gr(\mathbb{P}^2, \mathbb{P}^9))$, hence $Gr(\mathbb{P}^2, \mathbb{P}^9)$ is 2-tangentially weakly defective. Moreover since the $5$-contact locus has dimension 1, and the 5-tangentially contact locus has positive dimension, we can conclude that the 5-tangentially contact locus also has dimension 1.
 
 In the proof of Theorem \ref{Main} we have already computed with Macaulay2 that the $2$-tangentially contact locus of $Gr(\mathbb{P}^2, \mathbb{P}^6)$ is 0-dimensional. This suffices to prove that $Gr(\mathbb{P}^2, \mathbb{P}^6)$ is not 2-tangentially weakly defective.
\end{proof}

}

\section{Appendix on a Quantum Physical interpretation}
Quantum technologies are nowadays very active in giving a good measurement of the entanglement of the state of a quantum physical system. In particular systems of identical fermionic particles are of very high interest in Quantum Theory.

We would like to finish our paper with a physical interpretation of our containment diagram for the orbit closures of $SL(8)$ in $\mathbb{P}(\bigwedge^3\mathbb{C}^8)$.

From a physical point of view, an element of $\mathbb{P}(\bigwedge^3\mathbb{C}^8)$ can be interpreted as a system of 3 identical fermions, the state of each belonging to a 8-th dimensional ``~Hilbert~'' space. In \cite{BPRST} the authors describe how the entanglement of a state cannot change under Stochastic Local Quantum Operation and Classical Communication (SLOCC). Performing a SLOCC over a quantum system of $k$ identical fermionic particles on an $n$-dimensional vector space corresponds to act on a vector $|\phi \rangle \in \bigwedge^k \mathbb{C}^n$ with $GL(n)$. Then if one considers that the multiplication by scalars does not affect the state $|\phi\rangle$, one can operate with a SLOCC on the projective class of  $|\phi\rangle$ in  $\mathbb{P}(\bigwedge^k \mathbb{C}^n)$ remaining on the same orbit via $SL(n)$. For the states in the same orbit of $SL(n)$ the entanglement, according to \cite{BPRST}, does not change. Therefore the classification of all the orbits of $SL(n)$ in $\bigwedge^k\mathbb{C}^n$ gives a corresponding classification of all the possible ``~degrees of entanglement~" that a quantum state can have. To be more precise, the containment diagram of our Table \ref{cont} gives precisely the stratification of the entanglement measure of a 3 fermions in $\mathbb{C}^8$: in particular the variety II (i.e. the Grassmannian $Gr(\mathbb{P}^2, \mathbb{P}^7)$) there represents the pure separable states, and all the other states are more entangled as their level in the containment diagram is higher. For example points in the open part of $\sigma_2(Gr(\mathbb{P}^2, \mathbb{P}^7))$ (V) are more entangled than points in the open part of the restricted chordal variety III.

\section*{Acknowledgements} First of all we want to thank G. Ottaviani for the stimulating environment of the Numerical Algebraic Geometry working  group among Firenze and Bologna where the problem addressed in this paper was proposed and for his support in writing this paper. Moreover we like to thank L. Chiantini for many useful conversations and W.A. de Graaf for his help with GAP (\cite{GAP}). We also want to thank F. Han  and K. Ranestad for pointing out two mistakes in the first arXiv versions of this paper; a special thanks goes to K. Ranestad who helped a lot in finding out the 3-torsion scroll. The first author was partially supported by GNSAGA of INDAM, Prin Research Project GVA, Mathematical Department of Bologna and Mathematical Department of Trento.


\begin{thebibliography}{99}

\bibitem{AOP} H. Abo, G. Ottaviani and C. Peterson. 
Non-defectivity of Grassmannians of planes.
J. of Alg. Geom. 21 (2012) 1--20.

\bibitem{a1} M.F. Atiyah.
Complex fiber bundles and ruled surfaces.
Proc. London Math. Soc. 5 (1955) 407--434.

\bibitem{a} M.F. Atiyah.
Vector bundles over an elliptic curve. 
Proc. London Math. Soc. 7 (1957) 414--452.

\bibitem{adhpr} A.B. Aure, W. Decker, K. Hulek, S. Popescu and K. Ranestad,
The Geometry of Bielliptic Surfaces in $\mathbb{P}^4$. 
Internat. J. Math
4
(1993) 873--902.

\bibitem{A} B. \AA dlandsvik. 
Varieties with an extremal number of degenerate higher secant varieties. 
J. Reine Angew. Math. 392 (1988) 16--26.

\bibitem{Bertini} D.J. Bates, J.D. Hauenstein, A.J. Sommese and C.W. Wampler.
Bertini: software for numerical algebraic geometry.
Available at \url{bertini.nd.edu}.

\bibitem{BDG} K. Baur, J. Draisma and W. de Graaf.
Secant dimensions of minimal orbits: computations and conjectures.
Experimental Mathematics 16 (2007) 239--250.

\bibitem{BPRST} C.H. Bennet, S. Popescu, D. Rohrlich, J.A.Smolin, A.V. Thapliyal.
Exact and Asymptotic Measures of Multipartite Pure State Entanglement.
Phys. Rev. A 63, 012307 (2000).

\bibitem{bdhm} A. Bernardi, N.S. Daleo, J.D. Hauenstein, B. Mourrain.
Tensor decomposition and homotopy continuation.
Preprint: \url{arXiv:1512.04312}.

\bibitem{Ada} A. Boralevi. 
A note on secants of Grassmannians.
Rendiconti dell'Istituto Matematico dell'Universit\`a di Trieste 45 (2013) 67--72.

\bibitem{CGG} M.V. Catalisano, A.V. Geramita and A. Gimigliano. 
Secant varieties of Grassmann varieties.
Proc. Amer. Math. Soc. 133 (2005) 633--642.

\bibitem{CaCi} F. Catanese and C. Ciliberto.
Symmetric products of elliptic curves and surfaces of general type with $p_g=q=1$.
J. Algebraic Geometry 2 (1993) 389--411.

\bibitem{cc} L. Chiantini and C. Ciliberto.
Weakly defective varieties. 
Trans. Amer. Math. Soc. 354 (2002) 151--178.

\bibitem{cc2} L.Chiantini, C.Ciliberto
On the concept of k-th secant order of a variety.
J. London Math. Soc. 73 (2006) 436--454.

\bibitem{co} L. Chiantini, G. Ottaviani.
On generic identifiability of 3-tensors of small rank.
SIAM. J. Matrix. Anal. \& Appl. 33 (2012) 1018--1037.

\bibitem{cov} L. Chiantini, G. Ottaviani and N. Vannieuwenhoven. 
An algorithm for generic and low-rank specific identifiability of complex tensors. 
SIAM. J. Matrix Anal. \& Appl., 35 (2014) 1265--1287.

\bibitem{ch} C. Ciliberto and K. Hulek.
A Remark on the Geometry of Elliptic Scrolls and Bielliptic Surfaces.
Manuscripta Mathematica, 95 (1998),  213--224.


\bibitem{drag} D. $\check{\mathrm{Z}}$. Djokovi\'c.
Closures of equivalence classes of trivectors of an eight-dimensional complex vector space.
Canad. Math. Bull. Vol. 26 (1983) 92--100.

\bibitem{GAP} The GAP Group, GAP -- Groups, Algorithms, and Programming, Version 4.8.3; 2016. Available at: \url{http://www.gap-system.org}.

\bibitem{SLA} W.A. de Graaf,  {\sf SLA} - computing with {S}imple {L}ie {A}lgebras,
{a {\sf GAP}4 package,
 (2016).
Available at:  {\verb+(http://science.unitn.it/~degraaf/sla.html)+, 
version 1.1}.
}

\bibitem{m2} D.R. Grayson and M.E. Stillman.
Macaulay2, a software system for research in algebraic geometry.
Available at \url{http://www.math.uiuc.edu/Macaulay2/}.

\bibitem{gur1}  G.B. Gurevich.
Classification des trivecteurs ayant le rang huit.
P. Noordhoff Ltd. Groningen, the Netherlands, 1964.

\bibitem{gur}  G.B. Gurevich.
Foundations of the Theory of Algebraic Invariants.
Doklady Akademii nauk SSSR, Vol 2 (1935) 355-356.

\bibitem{hoos} J.D. Hauenstein, L. Oeding, G. Ottaviani, A. Sommese.
Homotopy techniques for tensor decomposition and perfect identifiability.
Preprint: \url{arXiv:1501.00090}.


\bibitem{jm} J.M. Landsberg.
Tensors: Geometry and Applications. Graduate studies in mathematics.
American Mathematical Soc.

\bibitem{Me} M. Mella.
Singularities of linear systems and the waring problem.
Trans. Amer. Math. Soc. 358 (2006) 5523--5538.

\bibitem{Se2} C. Segre.
Ricerche sulle rigate ellittiche di qualunque ordine.
Atti R. Accad. Torino 21 (1885-86), 628--651.

\bibitem{Se} C. Segre.
Sui complessi lineari di piani nello spazio a cinque dimensioni. 
Annali di Mat. pura ed applicata 27 (1917) 75--123.

\bibitem{Te} A. Terracini. 
Sulle $V_k$ per cui la variet\`a degli $S_h$ $(h+1)$-seganti ha dimensione minore dell'ordinario.
Rend. Circ. Mat. Palermo 31 (1911) 392--396.

\bibitem{Wa} K. Wakeford.
On Canonical Forms. 
Proc. London Math. Soc. 18 (1918--19) 403--410.

\bibitem{Za} O. Zariski.
Foundations of a general theory of birational correspondences.
Trans. Amer. Math. Soc. 53  (1943) 490--542.

\end{thebibliography}
\end{document}